\documentclass[12pt,a4paper]{article}
\usepackage{amsmath}
\usepackage{amsfonts}
\usepackage{amsthm}
\usepackage{amssymb}

\usepackage{enumerate}
\theoremstyle{plain}
\newtheorem{theo}{Theorem}[section]
\newtheorem{lem}[theo]{Lemma}
\newtheorem{prop}[theo]{Proposition}
\theoremstyle{definition}
\newtheorem{definition}[theo]{Definition}
\theoremstyle{remark}
\newtheorem{rem}[theo]{Remark}
\numberwithin{equation}{section}

\newcommand{\C}{\mathbb{C}}

\newcommand{\R}{\mathbb{R}}
\newcommand{\N}{\mathbb{N}}
\newcommand{\M}{\mathbb{M}}
\newcommand{\divrg}{\textrm{div}\,}

\title{Estimating area of inclusions in anisotropic plates {}from boundary
data
\thanks{Work supported by MIUR, PRIN no. 2003082352, no. 2004011204.}}
\author{Antonino Morassi\thanks{Dipartimento di Ingegneria Civile e Architettura,
Universit\`a degli Studi di Udine, via Cotonificio 114, 33100
Udine, Italy. E-mail: \textsf{antonino.morassi@uniud.it}} \ Edi
Rosset\thanks{Dipartimento di Matematica e Informatica,
Universit\`a degli Studi di Trieste, via Valerio 12/1, 34100
Trieste, Italy. E-mail: \textsf{rossedi@univ.trieste.it}} \ and
Sergio Vessella\thanks{DIMAD, Universit\`a degli Studi di Firenze,
Via Lombroso 6/17, 50134 Firenze, Italy,
\textsf{sergio.vessella@dmd.unifi.it}}}

\begin{document}

\maketitle

\noindent \textbf{Abstract.} We consider the inverse problem of
determining the possible presence of an inclusion in a thin plate
by boundary measurements. The plate is made by non-homogeneous
linearly elastic material belonging to a general class of
anisotropy. The inclusion is made by different elastic material.
Under some a priori assumptions on the unknown inclusion, we prove
constructive upper and lower estimates of the area of the unknown
defect in terms of an easily expressed quantity related to work,
which is given in terms of measurements of a couple field applied
at the boundary and of the induced transversal displacement and
its normal derivative taken at the boundary of the plate.
\medskip

\medskip

\noindent \textbf{Mathematical Subject Classifications (2000):}
35R30, 35R25, 73C02.

\medskip

\medskip

\noindent \textbf{Key words:} inverse problems, elastic plates,
inclusions, three spheres inequalities, size estimates, unique
continuation.

\section{Introduction} \label{sec:
introduction}

In this paper we consider an inverse problem in linear elasticity
consisting in the identification of an inclusion in a thin plate
by boundary measurements. Let $\Omega$ denote the middle plane of
the plate and let $h$ be its constant thickness. The inclusion $D$
is modelled as a plane subdomain compactly contained in $\Omega$.
Suppose we make the following diagnostic test. We take a
\textit{reference plate}, i.e. a plate without inclusion, and we
deform it by applying a couple field $\hat{M}$ at its boundary.
Let $W_0$ be the work exerted in deforming the specimen. Now, we
repeat the same experiment on a possibly defective plate. The
exerted work generally changes and assumes, say, the value $W$. In
this paper we want to find constructive estimates, {}from above
and {}from below, of the \textit{area} of the unknown inclusion
$D$ in terms of the difference $|W-W_0|$.

{}From the mathematical point of view, see \cite{l:fi}, \cite{l:gu} the
infinitesimal deformation of the defective plate is governed by
the fourth order Neumann boundary value problem
\begin{equation}
  \label{eq:intr.equation_with_D}
  \divrg(\divrg ((\chi_{\Omega \setminus {D}}{\mathbb P} + \chi_{D}
  \widetilde {\mathbb
P})\nabla^2 w))=0,
  \quad \mathrm{in}\ \Omega,
\end{equation}
\begin{equation}
  \label{eq:intr.bc1_with_D}
  ({\mathbb P} \nabla^2 w)n\cdot n=-\widehat {M}_n, \quad \mathrm{on}\ \partial
  \Omega,
\end{equation}
\begin{equation}
  \label{eq:intr.bc2_with_D}
  \divrg({\mathbb P} \nabla^2 w)\cdot n+(({\mathbb P} \nabla^2
  w)n\cdot \tau),_s
  =(\widehat M_\tau),_s, \quad \mathrm{on}\ \partial
  \Omega,
\end{equation}
where $w$ is the transversal displacement of the plate and
$\widehat M_{\tau}$, $\widehat M_n$ are the twisting and bending
components of the assigned couple field $\widehat M$,
respectively. In the above equations $\chi_D$ denotes the
characteristic function of $D$ and $n$, $\tau$ are the unit outer
normal and the unit tangent vector to $\partial \Omega$,
respectively. The plate tensors $\mathbb P$, $\widetilde {\mathbb
P}$ are given by
\begin{equation}
    \label{eq:PandC}
    \mathbb P = \frac{h^3}{12}\mathbb C, \quad \widetilde {\mathbb P} =
    \frac{h^3}{12} \widetilde {\mathbb C},
\end{equation}
where $\mathbb C$ is the elasticity tensor describing the response
of the material in the reference plate $\Omega$, whereas
$\widetilde {\mathbb C}$ denotes the (unknown) corresponding
tensor for the inclusion $D$. The work exerted by the couple field
$\widehat M$ has the expression
\begin{equation}
  \label{eq:intr.W}
  W=-\int_{\partial\Omega}\widehat M_{\tau,s}w+\widehat M_nw,_n.
  \end{equation}
When the inclusion $D$ is absent, the equilibrium problem
\eqref{eq:intr.equation_with_D}-\eqref{eq:intr.bc2_with_D} becomes
\begin{equation}
  \label{eq:4.equation_without D}
  \divrg(\divrg ( {\mathbb
P}\nabla^2 w_0))=0,
  \quad \mathrm{in}\ \Omega,
\end{equation}
\begin{equation}
  \label{eq:4.bc1_without_D}
  ({\mathbb P} \nabla^2 w_0)n\cdot n=-\widehat M_n, \quad \mathrm{on}\ \partial
  \Omega,
\end{equation}
\begin{equation}
  \label{eq:4.bc2_without_D}
  \divrg({\mathbb P} \nabla^2 w_0)\cdot n+(({\mathbb P} \nabla^2
  w_0)n\cdot \tau),_s
  =(\widehat M_\tau),_s, \quad \mathrm{on}\ \partial
  \Omega,
\end{equation}
where $w_0$ is the transversal displacement of the reference
plate. The corresponding external work exerted by $\widehat M$ is
given by
\begin{equation}
  \label{eq:intr.W_0}
  W_0=-\int_{\partial\Omega}\widehat M_{\tau,s}w_0+\widehat M_nw_{0,n}.
  \end{equation}
Our main result (see Theorem \ref{theo:4.main}) states that if,
for a given $h_1 >0$, the following \textit{fatness-condition}
\begin{equation}
    \label{eq:intr.fatness}
    area
    \left (
    \{x \in D | \ dist\{ x, \partial D\} > h_1 \}
    \right )
    \geq \frac{1}{2} area (D)
\end{equation}
holds, then
\begin{equation}
    \label{eq:intr.stime}
    C_1 \left |
    \frac{W-W_0}{W_0}
    \right |
    \leq
    area(D)
    \leq
    C_2 \left |
    \frac{W-W_0}{W_0}
    \right |,
\end{equation}
where the constants $C_1$, $C_2$ only depend on the a priori data.
Estimates \eqref{eq:intr.stime} are established under some
suitable ellipticity and regularity assumptions on the plate
tensor $\mathbb C$ and on the jump $\widetilde {\mathbb C} -
\mathbb C$.

Analogous bounds in plate theory were obtained in \cite{M-R-V1}
and \cite{M-R-V2} and recently in the context of shallow shells in
\cite{l:dlw}. The reader is referred to \cite{l:kss}, \cite{l:ar},
\cite{l:ars} for size estimates of inclusions in the context of
the electrical impedance tomography and to \cite{ikehata98},
\cite{l:amr03}, \cite{l:amr04}, \cite{l:amrv08} for corresponding
problems in two and three-dimensional linear elasticity. See also
\cite{l:LDiN09} for an application of the size estimates approach
in thermography. However, differently {}from \cite{M-R-V1} and
\cite{M-R-V2}, here we work under very general assumptions on the
constitutive properties of the reference plate, which is assumed
to be made by nonhomogeneous \textit{anisotropic} elastic material
satisfying the \textit{dichotomy condition}
\eqref{3.D(x)bound}--\eqref{3.D(x)bound 2} only. This choice
introduces significant difficulties in obtaining the upper bound
for $area(D)$, as we shall discuss shortly.

The first step of the proof of \textit{area estimates}
\eqref{eq:intr.stime} consists in proving that the strain energy
of the reference plate stored in the set $D$ is comparable with
the difference between the works exerted by the boundary couple
fields in deforming the plate with and without the inclusion. More
precisely, we have the following double inequality
\begin{equation}
    \label{eq:intr.energy-lemma}
    K_1
    \int_D | \nabla^2 w_0 |^2
    \leq
    |W-W_0|
    \leq K_2
    \int_D | \nabla^2 w_0 |^2,
\end{equation}
for suitable constants $K_1$, $K_2$ only depending on the a priori
data (see Lemma \ref{lem:7.2}). The proof of these bounds is based
on variational considerations and has been obtained in
\cite{M-R-V1} (Lemma 5.1).

The lower bound for $area(D)$ follows {}from the right hand side
of \eqref{eq:intr.energy-lemma} and {}from regularity estimates
for solutions to the fourth order elliptic equation
\eqref{eq:4.equation_without D} governing the equilibrium problem
in the anisotropic case.

In order to obtain the upper bound for $area(D)$ {}from the left
hand side of \eqref{eq:intr.energy-lemma}, the next issue is to
estimate {}from below $\int_D | \nabla^2 w_0 |^2$. This task is
rather technical and involves quantitative estimates of unique
continuation in the form of three spheres inequalities for the
hessian $\nabla^2 w_0$ of the reference solution $w_0$ to equation
\eqref{eq:4.equation_without D}. It is exactly to this point that
the dichotomy condition \eqref{3.D(x)bound}--\eqref{3.D(x)bound 2}
on the tensor $\mathbb C$ is needed. More precisely, it was shown
in \cite{M-R-V5} that if $\mathbb C$ satisfies the dichotomy
condition, then the plate operator of equation
\eqref{eq:4.equation_without D} can be written as the sum of a
product of two second order uniformly elliptic operators with
regular coefficients and a third order operator with bounded
coefficients. Then, Carleman estimates can be developed to derive
a three spheres inequality for $\nabla^2 w_0$ (see Theorem $6.2$
of \cite{M-R-V5}). The reader is referred to the paper
\cite{M-R-V5} for the necessary background.

The paper is organized as follows. Some basic notation is
introduced in Section \ref{sec:notation}. In Section 3 we state
the main result, Theorem \ref{theo:4.main}, which is proved in
Section 4. Section 5 is devoted to the proof of the
\textit{Lipschitz propagation of smallness property} (see
Proposition \ref{prop:6.1}), which is used in the proof of Theorem
\ref{theo:4.main}.

\section{Notation} \label{sec:notation}

We shall denote by $B_r(P)$ the disc in $\R^2$ of radius $r$ and
center $P$.

\noindent When representing locally a boundary as a graph, we use
the following notation. For every $x\in{\R}^2$ we set $x=(x_1,x_2)$,
where $x_1,x_2\in{\R}$.
\begin{definition}
  \label{def:2.1} (${C}^{k,1}$ regularity)
Let $\Omega$ be a bounded domain in ${\R}^{2}$. Given $k$, with
$k\in\N$, we say that a portion $S$ of $\partial \Omega$ is of
 \textit{class ${C}^{k,1}$ with constants $\rho_{0}$, $M_{0}>0$}, if,
for any $P \in S$, there exists a rigid transformation of
coordinates under which we have $P=0$ and
\begin{equation*}
  \Omega \cap B_{\rho_{0}}(0)=\{x=(x_1,x_2) \in B_{\rho_{0}}(0)\quad | \quad
x_2>\psi(x_1)
  \},
\end{equation*}
where $\psi$ is a ${C}^{k,1}$ function on
$\left(-\rho_0,\rho_0\right)$ satisfying
\begin{equation*}
\psi(0)=0,
\end{equation*}
\begin{equation*}
\nabla \psi (0)=0, \quad \hbox {when } k \geq 1,
\end{equation*}
\begin{equation*}
\|\psi\|_{{C}^{k,1}\left(-\rho_0,\rho_0\right)} \leq M_{0}\rho_{0}.
\end{equation*}

\medskip
\noindent When $k=0$, we also say that $S$ is of \textit{Lipschitz
class with constants $\rho_{0}$, $M_{0}$}.
\end{definition}
\begin{rem}
  \label{rem:2.1}
  We use the convention to normalize all norms in such a way that their
  terms are dimensionally homogeneous with their argument and coincide with the
  standard definition when the dimensional parameter equals one.
  For instance, given a function $u:\Omega\mapsto \R$, where $\partial
\Omega$ satisfies Definition \ref{def:2.1}, we denote
\begin{equation*}
  \|u\|_{{C}^{1,1}(\Omega)} =
  \|u\|_{{L}^{\infty}(\Omega)}+
  \rho_{0}\|\nabla u\|_{{L}^{\infty}(\Omega)}+
  {\rho_{0}}^{2}\|{\nabla}^{2} u\|_{{L}^{\infty}(\Omega)},
\end{equation*}
and
\begin{equation*}
\|u\|_{H^2(\Omega)}=\rho_0^{-1}\left(\int_\Omega u^2
+\rho_0^2\int_\Omega|\nabla u|^2 +\rho_0^4\int_\Omega|\nabla^2
u|^2\right)^{\frac{1}{2}},
\end{equation*}
and so on for boundary and trace norms such as
$\|\cdot\|_{H^{\frac{1}{2}}(\partial\Omega)}$,
$\|\cdot\|_{H^{-\frac{1}{2}}(\partial\Omega)}$.

\end{rem}

For any $r>0$ we denote
\begin{equation}
  \label{eq:2.int_env}
  \Omega_{r}=\{x \in \Omega \mid \hbox{dist}(x,\partial \Omega)>r
  \}.
\end{equation}
Given a bounded domain $\Omega$ in $\R^2$ such that $\partial
\Omega$ is of class $C^{k,1}$, with $k\geq 1$, we consider as
positive the orientation of the boundary induced by the outer unit
normal $n$ in the following sense. Given a point
$P\in\partial\Omega$, let us denote by $\tau=\tau(P)$ the unit
tangent at the boundary in $P$ obtained by applying to $n$ a
counterclockwise rotation of angle $\frac{\pi}{2}$, that is
\begin{equation}
    \label{eq:2.tangent}
        \tau=e_3 \times n,
\end{equation}
where $\times$ denotes the vector product in $\R^3$ and $\{e_1, e_2, e_3\}$
is the canonical basis in $\R^3$.

Given any connected component $\cal C$ of $\partial \Omega$ and
fixed a point $P_0\in\cal C$, let us define as positive the
orientation of $\cal C$ associated to an arclength
parameterization $\varphi(s)=(x_1(s), x_2(s))$, $s \in [0, l(\cal
C)]$, such that $\varphi(0)=P_0$ and
$\varphi'(s)=\tau(\varphi(s))$. Here $l(\cal C)$ denotes the
length of $\cal C$.

Throughout the paper, we denote by $w,_i$, $w,_s$, and $w,_n$ the
derivatives of a function $w$ with respect to the $x_i$ variable, to
the arclength $s$ and to the normal direction $n$, respectively, and
similarly for higher order derivatives.

We denote by $\M^2$ the space of $2 \times 2$ real valued matrices
and by ${\cal L} (X, Y)$ the space of bounded linear operators
between Banach spaces $X$ and $Y$.

For every pair of real $2$-vectors $a$ and $b$, we denote by $a
\cdot b$ the scalar product of $a$ and $b$.
For every $2 \times 2$ matrices $A$, $B$ and for every $\mathbb{L}
\in{\cal L} ({\M}^{2}, {\M}^{2})$, we use the following notation:
\begin{equation}
  \label{eq:2.notation_1}
  ({\mathbb{L}}A)_{ij} = L_{ijkl}A_{kl},\quad A \cdot B = A_{ij}B_{ij}, \quad |A|= (A \cdot A)^{\frac {1} {2}},
\end{equation}

\begin{equation}
  \label{eq:2.notation_3bis}
  A^{sym} =  \frac{1}{2} \left ( A + A^T \right ),
\end{equation}
where, here and in the sequel, summation over repeated indexes
is implied.

Moreover we say that
\begin{equation}
  \label{eq:2.notation_4}
  \widetilde {\mathbb L} \leq {\mathbb L},
\end{equation}
if and only if, for every $2 \times 2$ symmetric matrix $A$,
\begin{equation}
  \label{eq:2.notation_5}
  \widetilde {\mathbb L}A \cdot A \leq {\mathbb L}A \cdot A.
\end{equation}

\section{The main result} \label{sec:
direct}

Let us consider a thin plate
$\Omega\times[-\frac{h}{2},\frac{h}{2}]$ with middle surface
represented by a bounded domain $\Omega$ in $\R^2$ and having
uniform thickness $h$, $h<<\hbox{diam}(\Omega)$. We assume that
$\partial\Omega$ is of class $C^{1,1}$ with constants $\rho_0$,
$M_0$ and that, for a given positive number $M_1$, satisfies
\begin{equation}
  \label{eq:M_1}
  area(\Omega) \equiv |\Omega|\leq M_1\rho_0^2.
\end{equation}
We shall assume throughout that the elasticity tensor $\C$ of the
reference plate is known and has cartesian components $C_{ijkl}$
which satisfy the following symmetry conditions
\begin{equation}
  \label{eq:sym-conditions-C-components}
    C_{ijkl}(x)=C_{klij}(x)=C_{lkij}(x), \quad \hbox{ }{i,j,k,l=1,2},
     \hbox{ a.e. in } \Omega.
\end{equation}
On the elasticity tensor $\C$ let us make the following
assumptions:
\begin{enumerate}[i)]
\item \textit{Ellipticity (strong convexity)}

There exists a positive constant $\gamma$ such that
\begin{equation}
  \label{eq:3.convex}
    {\C}A \cdot A \geq \gamma |A|^2, \qquad \hbox{a.e. in } \Omega,
\end{equation}
for every $2\times 2$ symmetric matrix $A$.
\item \textit{$C^{1,1}$ regularity}

There exists $M>0$ such that
\begin{equation}
  \label{eq:4.regul}
  \sum_{i,j,k,l=1}^2 \sum_{m=0}^2 \rho_0^m \|\nabla^m C_{ijkl}\|_{L^\infty(\R^2)} \leq
    M.
\end{equation}

Condition \eqref{eq:sym-conditions-C-components} implies that
instead of $16$ coefficients we actually deal with $6$
coefficients and we denote
\begin{center}
\( {\displaystyle \left\{
\begin{array}{lr}
    C_{1111}=A_0, \ \ C_{1122}=C_{2211}=B_0,
         \vspace{0.12em}\\
    C_{1112}=C_{1121}=C_{1211}=C_{2111}=C_0,
        \vspace{0.12em}\\
    C_{2212}=C_{2221}=C_{1222}=C_{2122}=D_0,
        \vspace{0.12em}\\
    C_{1212}=C_{1221}=C_{2112}=C_{2121}=E_0,
        \vspace{0.12em}\\
    C_{2222}=F_0,
        \vspace{0.25em}\\
\end{array}
\right. } \) \vskip -3.0em
\begin{eqnarray}
\ & & \label{3.coeff6}
\end{eqnarray}
\end{center}
with
\begin{equation}
    \label{3.coeffsmall}
    a_0=A_0, \ a_1=4C_0, \ a_2=2B_0+4E_0, \ a_3=4D_0, \ a_4=F_0.
\end{equation}
Let $S(x)$ be the following $7\times 7$ matrix
\begin{equation}
    \label{3. S(x)}
    S(x) = {\left(
\begin{array}{ccccccc}
  a_0   & a_1   & a_2   & a_3   & a_4   & 0    &    0    \\
  0     & a_0   & a_1   & a_2   & a_3   & a_4  &    0    \\
  0     & 0     & a_0   & a_1   & a_2   & a_3  &    a_4  \\
  4a_0  & 3a_1  & 2a_2  & a_3   & 0     & 0    &    0    \\
  0     & 4a_0  & 3a_1  & 2a_2  & a_3   & 0    &    0    \\
  0     & 0     & 4a_0  & 3a_1  & 2a_2  & a_3  &    0    \\
  0     & 0     & 0     & 4a_0  & 3a_1  & 2a_2 &    a_3  \\
  \end{array}
\right)},
\end{equation}
and
\begin{equation}
    \label{3.D(x)}
    {\mathcal{D}}(x)= \frac{1}{a_0} |\det S(x)|.
\end{equation}
On the elasticity tensor $\mathbb{C}$ we make the following
additional assumption:

\item \textit{Dichotomy condition}
\begin{subequations}
\begin{eqnarray}
\label{3.D(x)bound} either &&  {\mathcal{D}}(x)>0,
\quad\hbox{for every } x\in \R^2, \\[2mm]
\label{3.D(x)bound 2} or && {\mathcal{D}}(x)=0, \quad\hbox{for
every } x\in \R^2,
\end{eqnarray}
\end{subequations}
where ${\mathcal{D}}(x)$ is defined by \eqref{3.D(x)}.
\begin{rem}
  \label{rem:dichotomy} Whenever \eqref{3.D(x)bound} holds we denote
\begin{equation}
    \label{delta-1}
    \mu=\min_{\R^2}{\mathcal{D}}.
\end{equation}
We emphasize that, in all the following statements, whenever a
constant is said to depend on $\mu$ (among other quantities) it is
understood that such dependence occurs \textit{only} when
\eqref{3.D(x)bound} holds.
\end{rem}
\end{enumerate}
Let $D\times [- \frac{h}{2}, \frac{h}{2}]$ be a possible unknown
inclusion in the plate, where $D$ is a measurable, possibly
disconnected subset of $\Omega$ satisfying
\begin{equation}
  \label{eq:4.d_0}
  \hbox{dist}(D, \partial \Omega) \geq d_{0}\rho_0,
\end{equation}
for some positive constant $d_0$.

Concerning the material forming the inclusion, we assume that the
corresponding elasticity tensor $\widetilde{{\mathbb
C}}=\widetilde{{\mathbb C}}(x)$ belongs to $L^\infty(\Omega,{\cal
L} ({\M}^{2}, {\M}^{2}))$ and has Cartesian components which
satisfy the symmetry conditions
\begin{equation}
  \label{eq:sym-conditions-Ctilde-components}
    \widetilde{C}_{ijkl}(x)=\widetilde{C}_{klij}(x)=\widetilde{C}_{lkij}(x), \quad \hbox{ }{i,j,k,l=1,2},
     \hbox{ a.e. in } \Omega.
\end{equation}
Moreover, we assume the following \textit{jump conditions} on
$\widetilde {\mathbb C}$: either there exist $\eta_0>0$ and
$\eta_1>1$ such that
\begin{equation}
  \label{eq:4.jump+}
  \eta_0{\mathbb C} \leq \widetilde {\mathbb C}-{\mathbb C} \leq (\eta_1-1){\mathbb
  C},
  \quad \hbox{a.e. in } \Omega,
  \end{equation}
\noindent or there exist $\eta_0>0$ and $0<\eta_1<1$ such that
\begin{equation}
  \label{eq:4.jump-}
  -(1-\eta_1){\mathbb C} \leq \widetilde {\mathbb C}-{\mathbb C} \leq -\eta_0{\mathbb
  C},
  \quad \hbox{a.e. in }\Omega.
\end{equation}
Let us assume that the body forces inside the plate are absent and
that a couple field $\hat{M}$ is acting on the boundary of
$\Omega$. We shall assume:
\begin{equation}
  \label{eq:emme-reg+}
  \widehat M\in L^2(\partial\Omega,\R^2),
\end{equation}
\begin{equation}
  \label{eq:emme-supp1}
  \hbox{supp}(\widehat M)\subset\Gamma,
\end{equation}
where $\Gamma$ is an open subarc of $\partial\Omega$, such that
\begin{equation}
  \label{eq:emme-supp2}
  |\Gamma|\leq (1-\delta_0)|\partial\Omega|,
\end{equation}
for some positive constant $\delta_0$. Moreover, we obviously
assume the \textit{compatibility conditions} on the boundary
couple field $\widehat {M}$
\begin{equation}
  \label{eq:emme-comp-cond}
  \int_{\partial\Omega}\widehat M_\alpha=0,\qquad\alpha=1,2,
\end{equation}
and that, for a given constant $F>0$,
\begin{equation}
  \label{eq:emme-frequency}
  \frac{\|\widehat M \|_{L^2(\partial
\Omega,\R^2)}}{\|\widehat M \|_{H^{-\frac{1}{2}}(\partial
\Omega,\R^2)}}\leq F.
\end{equation}

Let us notice that, following a standard convention in the theory of plates, we
represent the boundary couple field $\widehat M$ in cartesian
coordinates as
\begin{equation}
  \label{eq:emme-cartesiane}
  \widehat M= \widehat M_2 e_1+\widehat M_1 e_2, \qquad \hbox{on }
  \partial\Omega.
\end{equation}
The equilibrium problem of the plate with and without inclusion is
described by the Neumann problem
\eqref{eq:intr.equation_with_D}-\eqref{eq:intr.bc2_with_D} and
\eqref{eq:4.equation_without D}-\eqref{eq:4.bc2_without_D},
respectively. Under the above assumptions, the problems
\eqref{eq:intr.equation_with_D}-\eqref{eq:intr.bc2_with_D} and
\eqref{eq:4.equation_without D}-\eqref{eq:4.bc2_without_D} have
solutions  $w\in H^2(\Omega)$, $w_0\in H^2(\Omega)$, respectively.
These solutions are uniquely determined by imposing the
normalization conditions
\begin{equation}
  \label{eq:normalization-w}
  \int_\Omega w=0, \qquad \int_\Omega w,_\alpha=0,\qquad\alpha=1,2,
\end{equation}
\begin{equation}
  \label{eq:normalization-w0}
  \int_\Omega w_0=0, \qquad \int_\Omega
  w_0,_\alpha=0,\qquad\alpha=1,2.
\end{equation}
We recall that the quantities $W$, $W_0$ defined by
\eqref{eq:intr.W}, \eqref{eq:intr.W_0} represent the work exerted
by the boundary value couple field $\hat{M}$ when the inclusion
$D$ is present or absent, respectively. By the weak formulation of
problems
\eqref{eq:intr.equation_with_D}--\eqref{eq:intr.bc2_with_D} and
\eqref{eq:4.equation_without D}--\eqref{eq:4.bc2_without_D}, the
works $W$ and $W_0$ coincide with the strain energies stored in
the plate, namely
\begin{equation}
     \label{eq:energy-with-D}
    W=\int_\Omega ( \chi_{\Omega \setminus D} \mathbb P + \chi_D
    \widetilde{\mathbb P}) \nabla^2 w \cdot \nabla^2 w,
\end{equation}
\begin{equation}
     \label{eq:energy-without-D}
    W_0=\int_\Omega \mathbb P \nabla^2 w_0 \cdot \nabla^2
    w_0.
\end{equation}
We are now in position to state the main result of this paper.
\begin{theo}
  \label{theo:4.main}
  Let $\Omega$ be a bounded domain in $\R^{2}$, such that $\partial \Omega$
is of
  class ${C}^{2,1}$ with constants $\rho_{0}, M_{0}$ and satisfying \eqref{eq:M_1}. Let $D$ be a
  measurable subset of $\Omega$ satisfying \eqref{eq:4.d_0} and
\begin{equation}
  \label{eq:4.fat}
  \left| D_{h_{1}\rho_0} \right| \geq \frac {1} {2} \left| D \right|,
  \end{equation}
for a given positive constant $h_{1}$. Let ${\mathbb P}$ given by
\eqref{eq:PandC} satisfy \eqref{eq:sym-conditions-C-components},
\eqref{eq:3.convex}, \eqref{eq:4.regul} and the dichotomy
condition \eqref{3.D(x)bound}--\eqref{3.D(x)bound 2}. Let
$\widetilde {\mathbb P}\in L^\infty(\Omega,{\cal L} ({\M}^{2},
{\M}^{2}))$, defined by \eqref{eq:PandC}, satisfy
\eqref{eq:sym-conditions-Ctilde-components}. Let $\widehat M \in
L^2(\partial \Omega, \R^2)$ satisfy \eqref{eq:emme-supp1}--\eqref{eq:emme-frequency}. If
\eqref{eq:4.jump+} holds, then we have
\begin{equation}
  \label{eq:4.size_estimate+}
  \frac {1} {\eta_1-1} C^{+}_{1}\rho_0^2
  \frac{W_0-W}{W_0}
  \leq
  |D|
  \leq
  \frac {\eta_1} {\eta_0} C^{+}_{2}\rho_0^2
  \frac{W_0-W}{W_0}
  .
  \end{equation}
If, conversely, \eqref{eq:4.jump-} holds, then we have
\begin{equation}
  \label{eq:4.size_estimate-}
  \frac {\eta_1} {1-\eta_1} C^{-}_{1}\rho_0^2
  \frac{W-W_0}{W_0}
  \leq
  |D|
  \leq
  \frac {1} {\eta_0} C^{-}_{2}\rho_0^2
  \frac{W-W_0}{W_0}
  ,
  \end{equation}
where $C^{+}_{1}$, $C^{-}_{1}$ only depend on $h$, $M_{0}$, $M_1$,
$d_{0}$, $\gamma$, $\mu$, $M$, whereas $C^{+}_{2}$, $C^{-}_{2}$
only depend on the same quantities and also on $\delta_0$, $h_{1}$
and $F$.
\end{theo}

\section{Proof of Theorem \ref{theo:4.main} } \label{sec:
proof_main}

The proof of Theorem \ref{theo:4.main} is mainly based on the
following key ingredients: energy estimates for the equilibrium
problems
\eqref{eq:intr.equation_with_D}--\eqref{eq:intr.bc2_with_D} and
\eqref{eq:4.equation_without D}--\eqref{eq:4.bc2_without_D} (Lemma
\ref{lem:7.2}) and an estimate of continuation {}from the interior
for solutions to the Neumann problem \eqref{eq:4.equation_without
D}--\eqref{eq:4.bc2_without_D} (Proposition \ref{prop:6.1}).

\begin{lem}
  \label{lem:7.2}
  Let the fourth-order tensor fields ${\mathbb P},\widetilde {\mathbb
  P}\in L^\infty(\Omega,{\cal L} ({\M}^{2}, {\M}^{2}))$
  given by
  \eqref{eq:PandC},
  satisfy the symmetry conditions
  \eqref{eq:sym-conditions-C-components} and
  \eqref{eq:sym-conditions-Ctilde-components}, respectively. Let
  $\widehat M\in H^{-\frac{1}{2}}(\partial\Omega,\R^2)$ satisfy \eqref{eq:emme-comp-cond}.
  Let
  $\xi_{0}$, $\xi_{1}$, $0<\xi_{0}<\xi_{1}$, be such that
\begin{equation}
  \label{eq:7.13}
  \xi_{0} |A|^{2} \leq {\mathbb P}(x)A \cdot A \leq
  \xi_{1}|A|^{2},
\quad \quad for
  \quad a.e. \quad x \in \Omega,
\end{equation}
for every symmetric matrix $A\in {\M}^2$, and let the jump
$(\widetilde {\mathbb P}(x) - {\mathbb P}(x))$ satisfy either
\eqref{eq:4.jump+} or \eqref{eq:4.jump-}. Let $w$, $w_{0} \in
H^{2}(\Omega)$ be the weak solutions to the problems
  \eqref{eq:intr.equation_with_D}--\eqref{eq:intr.bc2_with_D},
  \eqref{eq:4.equation_without D}--\eqref{eq:4.bc2_without_D} respectively.

If \eqref{eq:4.jump+} holds, then we have
\begin{equation}
  \label{eq:7.14+}
  \frac {\eta_0\xi_{0}} {\eta_1} \int_{D}| {\nabla^2} w_{0}|^{2}
  \leq
  W_0-W
  \leq
  (\eta_1-1)\xi_{1} \int_{D}| {\nabla^2} w_{0}|^{2}
 .
\end{equation}
If, instead, \eqref{eq:4.jump-} holds, then we have
\begin{equation}
  \label{eq:7.14-}
  \eta_0 \xi_{0} \int_{D}| {\nabla^2} w_{0}|^{2}
  \leq
  W-W_0
  \leq
  \frac {1-\eta_1}{\eta_1} \xi_{1} \int_{D}| {\nabla^2} w_{0}|^{2}
  .
\end{equation}
\end{lem}

The proof of the above lemma is given in \cite{M-R-V1}, Lemma 5.1.

\begin{prop} [{\bf Lipschitz propagation of smallness}]
\label{prop:6.1} Let $\Omega$ be a bounded domain in $\R^{2}$,
such that $\partial \Omega$ is of class ${C}^{2,1}$ with constants
$\rho_{0}, M_{0}$ and satisfying \eqref{eq:M_1}. Let the fourth order tensor $\mathbb{P}$ be
defined by \eqref{eq:PandC} and satisfying
\eqref{eq:sym-conditions-C-components}, \eqref{eq:3.convex},
\eqref{eq:4.regul} and the dichotomy condition
\eqref{3.D(x)bound}--\eqref{3.D(x)bound 2}. Let $w_0 \in
H^{2}(\Omega)$ be the unique weak solution of the problem
\eqref{eq:4.equation_without D}--\eqref{eq:4.bc2_without_D}
satisfying \eqref{eq:normalization-w0}, with $\widehat M \in
L^2(\partial \Omega,\R^2)$ satisfying \eqref{eq:emme-supp1}--\eqref{eq:emme-frequency}.
There exists
$s>1$, only depending on $\gamma$, $M$, $\mu$, $M_0$ and
$\delta_0$, such that for every $\rho>0$ and every $\bar x\in
\Omega_{s\rho}$, we have

\begin{equation}
   \label{eq:LPS}
\int_{B_\rho(\bar x)}|\nabla^2 w_0|^2\geq
\frac{C}{\exp\left[A\left(\frac{\rho_0}{\rho}\right)^B\right]}
\int_{\Omega}|\nabla^2 w_0|^2,
\end{equation}
where $A>0$, $B>0$ and $C>0$ only depend on $h$, $M_{0}$, $M_1$,
$\gamma$, $\mu$, $M$, $\delta_0$
and $F$.
\end{prop}

\begin{proof}[Proof of Theorem \ref{theo:4.main}]
By the hypotheses made on ${\mathbb P}$, the inequality
\eqref{eq:7.13} is satisfied with $\xi_{0}=\gamma\frac{h^3}{12}$,
$\xi_{1}=\frac{h^3}{6}M$, so that Lemma \ref{lem:7.2} can be
applied.

By standard interior regularity estimates (see, for instance,
Theorem 8.3 in \cite{M-R-V1}) and by the Sobolev embedding theorem,
we have
\begin{equation}
  \label{eq:7.24}
  \| {\nabla}^2 w_{0}\|_{{L}^{\infty}(D)}
  \leq
  \frac{C}{\rho_0^2}\|w_{0}\|_{{H}^{2}(\Omega)},
\end{equation}
with $C$ only depending on  $\gamma$, $h$, $M$ and $d_0$.

{}From \eqref{eq:7.24}, Poincar\'{e} inequality, \eqref{eq:7.13},
\eqref{eq:energy-without-D}, we have
\begin{equation}
  \label{eq:7.24bis}
  \| {\nabla}^2 w_{0}\|_{{L}^{\infty}(D)}
  \leq
  \frac{C}{\rho_0} W_0^{\frac {1}{2}},
\end{equation}
where the constant $C$ only depends on $\gamma$, $h$, $M$, $d_{0}$,
$M_0$ and $M_1$.

The lower bound for $|D|$ in \eqref{eq:4.size_estimate+},
\eqref{eq:4.size_estimate-} follows {from} the right hand side of
\eqref{eq:7.14+}, \eqref{eq:7.14-} and {from} \eqref{eq:7.24bis}.

Next, let us prove the upper bound for $|D|$ in
\eqref{eq:4.size_estimate+}, \eqref{eq:4.size_estimate-}.

Let $\epsilon=\min \{ \frac {2 d_{0}}{s}, \frac {h_{1}}{\sqrt {2}}
\}$, where $s$ is as in Proposition \ref{prop:6.1}. Let us cover
$D_{h_{1}\rho_0}$ with internally non overlapping closed squares
$Q_{l}$ of side $\epsilon\rho_0$, for $l=1,...,L$. By the choice
of $\epsilon$ the squares $Q_{l}$ are contained in $D$. Let $\bar
{l}$ be such that $\int_{Q_{\bar {l}}} {| {\nabla}^2
w_{0}|}^{2}=\min_{l}\int_{Q_{l}} {| {\nabla}^2 w_{0}|}^{2}$.
Noticing that $|D_{h_{1}\rho_0}|\leq L\epsilon^2\rho_0^2$, we have
\begin{equation}
  \label{eq:7.25}
  \int_{D} {| {\nabla}^2 w_{0}|}^{2}
  \geq
  \int_{\bigcup_{l=1}^{L} Q_{l}} {| {\nabla}^2 w_{0}|}^{2}
  \geq
  L\int_{Q_{\bar {l}}} {| {\nabla}^2 w_{0}|}^{2}
  \geq
  \frac {|D_{h_{1}\rho_0}|}{\rho_0^2\epsilon ^{2}}
  \int_{Q_{\bar {l}}} {| {\nabla}^2 w_{0}|}^{2}.
\end{equation}
Let $\bar {x}$ be the center of $Q_{\bar {l}}$. {From}
\eqref{eq:7.13}, \eqref{eq:7.25}, estimate \eqref{eq:LPS} with
$\rho=\frac{\epsilon}{2}\rho_0$, {from}
\eqref{eq:energy-without-D} and by our hypothesis \eqref{eq:4.fat}
we have
\begin{equation}
  \label{eq:7.26}
  \int_{D} {| {\nabla}^2 w_{0}|}^{2}
  \geq
  \frac{K |D|}{\rho_0^2} W_0,
\end{equation}
where $K$ is a positive constant only depending on $\gamma$, $h$,
$M$, $M_0$, $M_1$, $d_{0}$, $\delta_0$, $h_{1}$ and $F$. The upper bound for $|D|$ in
\eqref{eq:4.size_estimate+}, \eqref{eq:4.size_estimate-} follows
{from} the left hand side of \eqref{eq:7.14+},\eqref{eq:7.14-} and
{from} \eqref{eq:7.26}.
\end{proof}

\section{Proof of Proposition \ref{prop:6.1}} \label{sec:
LPS}

Let us premise the following Lemmas.
\begin{prop} [Three Spheres Inequality]
\label{prop:3spheres} Let $\Omega$ be a domain in $\R^2$, and let
the plate tensor $\mathbb{P}$ given by \eqref{eq:PandC} satisfies
\eqref{eq:sym-conditions-C-components}, \eqref{eq:3.convex},
\eqref{eq:4.regul} and the dichotomy condition
\eqref{3.D(x)bound}--\eqref{3.D(x)bound 2}. Let $u\in H^2(\Omega)$
be a weak solution to the equation
\begin{equation}
  \label{eq:plate_eq}
{\rm div}({\rm div} (
      {\mathbb P}\nabla^2 u))=0,
       \quad \hbox{in }\Omega.
\end{equation}
For every $r_1, r_2, r_3, \overline{r}$, $0<r_1<r_2<r_3\leq
\overline{r}$, and for every $x\in \Omega_{\overline{r}}$ we have
\begin{equation}
  \label{eq:3sph}
   \int_{B_{r_{2}}(x)}|\nabla ^2 u|^{2} \leq C
   \left(  \int_{B_{r_{1}}(x)}|\nabla ^2 u|^{2}
   \right)^{\delta}\left(  \int_{B_{r_{3}}(x)}|\nabla ^2 u|^{2}
   \right)
   ^{1-\delta},
\end{equation}
where $C>0$ and $\delta$, $0<\delta<1$, only depend on $\gamma$,
$M$, $\mu$, $\frac{r_{3}}{r_{2}}$ and $\frac{r_{3}}{r_{1}}$.
\end{prop}
A proof of the above proposition can be easily obtained by Theorem
6.5 in \cite{M-R-V5}.

In order to prove Proposition \ref{prop:6.1}, we need the estimate
stated in the following Lemma (for the proof see \cite{M-R-V1}, Lemma 7.1).
\begin{lem}
\label{lem:6.1} Let $\Omega$ be a bounded domain in $\R^{2}$, such
that $\partial \Omega$ is of class ${C}^{2,1}$ with constants
$\rho_{0}, M_{0}$. Let the fourth order tensor $\mathbb{P}$ be
defined by \eqref{eq:PandC} and satisfying
\eqref{eq:sym-conditions-C-components}, \eqref{eq:3.convex} and
\eqref{eq:4.regul}. Let $w_0 \in H^{2}(\Omega)$ be the unique weak
solution of the problem \eqref{eq:4.equation_without
D}--\eqref{eq:4.bc2_without_D} satisfying
\eqref{eq:normalization-w0}, with $\widehat M \in H^{-
\frac{1}{2}}(\partial \Omega,\R^2)$ satisfying
\eqref{eq:emme-supp1}--\eqref{eq:emme-comp-cond}. We have
\begin{equation}
    \label{6.M_hat_above}
\|\widehat M\|_{H^{-\frac{1}{2}}(\partial \Omega,\R^2)}\leq
C\|\nabla ^2w_0\|_{L^2(\Omega)},
\end{equation}
where $C$ is a positive constant only depending on $M_0$, $M_1$,
$\delta_0$ and $M$.
\end{lem}
\begin{lem}
  \label{lem:Frequency}
Let the hypotheses of Proposition \ref{prop:6.1} be satisfied.
There exists $\widetilde{\rho}>0$, only depending on $M_0$, $M_1$,
$\delta_0$, $\gamma$, $M$, $\mu$ and $F$, such that for every
$r\leq\widetilde{\rho}$ we have
\begin{equation}
  \label{eq:Frequency}
   \frac{\int_{\Omega_r}|\nabla^2 w_0|^2}{\int_{\Omega}|\nabla^2 w_0|^2}\geq \frac{1}{2}.
\end{equation}
\end{lem}
\begin{proof}
Let us set
\begin{equation}
  \label{eq:riscrittura}
   \frac{\int_{\Omega_r}|\nabla^2 w_0|^2}{\int_{\Omega}|\nabla^2 w_0|^2}
   = 1-\frac{\int_{\Omega\setminus\Omega_r}|\nabla^2 w_0|^2}{\int_{\Omega}|\nabla^2 w_0|^2}
   .
\end{equation}
By H\"{o}lder inequality
\begin{equation}
  \label{eq:holder}
   \|\nabla^2 w_0\|^2_{L^2\left({\Omega\setminus\Omega_r}\right)}\leq
   \left|{\Omega\setminus\Omega_r}\right|^{\frac{1}{2}}
   \|\nabla^2 w_0\|^2_{L^4\left({\Omega\setminus\Omega_r}\right)},
\end{equation}
and by Sobolev inequality \cite{l:ad}
\begin{equation}
  \label{eq:sobolev}
   \|\nabla^2 w_0\|^2_{L^4\left(\Omega\right)}\leq C \|\nabla^2 w_0\|^2_{H^{\frac{1}{2}}(\Omega)},
\end{equation}
we have
\begin{equation}
  \label{eq:hol-sob}
   \|\nabla^2 w_0\|^2_{L^2({\Omega\setminus\Omega_r})}\leq
   \frac{C}{\rho_0^4} \left|{\Omega\setminus\Omega_r}\right|^{\frac{1}{2}}
   \| w_0\|^2_{H^{\frac{5}{2}}(\Omega)},
\end{equation}
where $C$ only depends on $M_0$, $M_1$. We recall that, by the
variational formulation of the problem
\eqref{eq:4.equation_without D}--\eqref{eq:4.bc2_without_D}, the
function $w_0$ satisfies
\begin{equation}
  \label{eq:stab-0}
    \|w_0\|_{H^2(\Omega)}
    \leq C\rho_0^2
    \|\widehat M\|_{H^{- \frac{1}{2}}(\partial \Omega, \R^2)},
\end{equation}
where $C>0$ only depends on $h$, $M_0$, $M_1$ and $\gamma$. Now, by
using the following regularity estimate (see \cite{M-R-V4} for a
proof)
\begin{equation}
  \label{eq:regularity-H4}
  \|w_0\|_{H^3(\Omega)}\leq C\rho_0^2\|\widehat M\|_{H^{\frac{1}{2}}(\partial\Omega, \R^2)},
\end{equation}
where $C>0$ only depends on $h$, $M_0$, $M_1$, $\gamma$ and $M$.
By interpolating \eqref{eq:stab-0} and
\eqref{eq:regularity-H4}, we get
\begin{equation}
  \label{eq:5/2_2}
   \| w_0\|_{H^{\frac{5}{2}}(\Omega)}\leq C\rho_0^2\|\widehat{M}\|_{L^2(\partial\Omega,\R^2)},
\end{equation}
where $C$ only depends on $h$, $M_0$, $M_1$, $\gamma$ and $M$.

Moreover
\begin{equation}
  \label{eq:AR}
   \left|\Omega\setminus\Omega_r\right|\leq Cr,
\end{equation}
with $C$ only depending on $M_0$ and $M_1$, see for details (A.3) in
\cite{l:ar}. {}From \eqref{eq:hol-sob},
\eqref{eq:5/2_2} and \eqref{eq:AR} we have
\begin{equation}
  \label{eq:stima_num}
   \int_{\Omega\setminus\Omega_r}|\nabla^2 w_0|^2\leq
   C\rho_0^2r^{\frac{1}{2}}\|\widehat{M}\|_{L^2(\partial\Omega,\R^2)},
\end{equation}
where $C$ only depends on $M_0$, $M_1$, $\gamma$, $M$. Finally, by
\eqref{eq:riscrittura}, \eqref{eq:stima_num}
and \eqref{6.M_hat_above} we obtain \eqref{eq:Frequency}.
\end{proof}

\begin{proof} [Proof of Proposition \ref{prop:6.1}]
It is not restrictive to assume $\rho_0=1$.

Set
\begin{equation*}
\vartheta_0=\arctan\frac{1}{M_0},
\end{equation*}
\begin{equation*}
s=\frac{5+\sin\vartheta_0+\sqrt{\sin^2\vartheta_0
+30\sin\vartheta_0+25}}{2\sin\vartheta_0},
\end{equation*}
\begin{equation*}
\chi=\frac{s\sin\vartheta_0}{5}=
\frac{5+\sin\vartheta_0+\sqrt{\sin^2\vartheta_0
+30\sin\vartheta_0+25}}{10},
\end{equation*}
\begin{equation*}
\vartheta_1=\arcsin\frac{1}{s}.
\end{equation*}
Let us notice that $s>1$, $\chi>1$ and $\vartheta_1>0$ only depend
on $M_0$.

Given $z\in{\R}^2$,
$\xi\in{\R}^2$, $|\xi|=1$, $\vartheta>0$, we shall denote by
\begin{equation}
\label{eq:cono}
     C(z,\xi,\vartheta)=\{x\in{\R}^2 \hbox{ s. t. }
      \frac{(x-z)\cdot\xi}{|x-z|}>\cos\vartheta\},
\end{equation}
the open cone having vertex $z$, axis in the direction $\xi$ and
width $2\vartheta$.

\medskip
\noindent {\bf Step 1} \textit{For every $\rho$, $0<\rho\leq
\rho_1=\frac{1}{16s}$,  and for every $x\in\Omega$ satisfying
$s\rho<\hbox{dist}(x,\partial\Omega)\leq\frac{1}{4}$, there exists
$\tilde x\in\Omega$ such that
\item{i)} $B_{5\chi\rho}(x)\subset C(\tilde
x,\frac{x-\tilde x}{|x-\tilde x|},\vartheta_0)\cap
B_{\frac{{1}}{8}}(\tilde x)\subset\Omega$,
\item{ii)} the discs $B_{\rho}(x)$ and $B_{\chi\rho}(x_2)$ are
internally tangent to $C(\tilde x,\frac{x-\tilde x}{|x-\tilde
x|},\vartheta_1)$, where $x_2=x+(\chi+1)\rho\frac{x-\tilde
x}{|x-\tilde x|}$.}

The proof of this step has merely geometrical character and has been
given in \cite{l:mr04}, Proof of Proposition 3.1. Up
to a rigid motion, we may assume that
$\frac{x-\widetilde{x}}{\left|x-\widetilde{x}\right|}=e_2$, where
$(e_1,e_2)$ is the canonical basis of $\mathbb{R}^2$

Set
\begin{equation*}
\begin{array}{lll} r_1=\rho,\phantom{aaaaa}
&r_k=\chi r_{k-1}=\chi^{k-1}\rho, &k\geq 2,\\
x_1=x,\phantom{aaaaa} &x_k=x_{k-1}+(r_{k-1}+r_k)e_2, &k\geq 2.\\
\end{array}
\end{equation*}

For every $k\in\N$, $B_{r_{k}}(x_k)$ is internally tangent to the
cone $C(\tilde x,e_2,\vartheta_1)$ and $B_{5\chi r_k}(x_k)$ is
internally tangent to the cone $C(\tilde x,e_2,\vartheta_0)$.
Moreover, we have that $B_{5 r_{k}}(x_k)\subset
B_{\frac{{1}}{8}}(\tilde x)$ if and only if
\begin{equation}
   \label{eq:upper_k-1}
   k-1\leq\frac{\log\left\{\frac{\chi-1}
   {6\chi-4}
   \left(\frac{{1}}{8\rho}-s+1+\frac{2}{\chi-1}
   \right)\right\}}{\log\chi}.
\end{equation}
In order to ensure that $B_{5 r_k}(x_k)\subset
B_{\frac{{1}}{8}}(\tilde x)$ holds at least for $k=1,2$, let us
assume also that $\rho\leq\rho_2=\frac{1} {8(6\chi+s+1)}$. Let
us define
\begin{equation}
   \label{eq:def_k_of_rho}
   k(\rho)=\left[\frac{\log\left\{\frac{\chi-1}
   {6\chi-4}\left(
   \frac{{h_0}}{8\rho}-s+1+\frac{2}{\chi-1}
   \right)\right\}}{\log\chi}\right]+1,
\end{equation}
where $h_0$, $0<h_0<1$, only depending on $M_0$, is
such that $\Omega_h$ is connected for every $h<h_0$ (see Prop. 5.5
in \cite{A-R-R-V}) and $[\ \cdot\ ]$ denotes the integer part of a real
number. We have that $B_{5r_{k(\rho)}}(x_{k(\rho)})\subset
B_{\frac{1}{8}}(\tilde x)\cap\Omega$ and $B_{5\chi r_j}(x_j)\subset
B_{\frac{1}{8}}(\tilde x)\cap\Omega$ for every $j=1,...,k(\rho)-1$.

Moreover let $\rho\leq\rho_3=\frac{h_0}{16s}$. We have
\begin{equation}
   \label{eq:lower_k}
   k(\rho)\geq\frac{\log\frac{\tau}{\rho}}{\log\chi},
\end{equation}
with $\tau=\frac{(\chi-1)h_0} {16\left(6\chi-4\right)}$.
Assuming also that $\rho\leq\rho_4=\frac{(\chi-1)h_0}{16}$, and
noticing that $\frac{\chi-1} {6\chi-4}\leq \frac{1}{5}$, we have
\begin{equation}
   \label{eq:upper_k}
   k(\rho)\leq\frac{\log\frac{h_0}{20\rho}}{\log\chi}+1.
\end{equation}
{}From \eqref{eq:lower_k} and \eqref{eq:upper_k}, it follows that,
for $\rho\leq\bar\rho=\min\{\rho_1,\rho_2,\rho_3, \rho_4\}$,
\begin{equation}
   \label{eq:upper_lower_rho_k}
   \frac{\tau}{\chi}\leq r_{k(\rho)}=
   \chi^{k(\rho)-1}\rho\leq\frac{h_0}{20}.
\end{equation}

\medskip
\noindent {\bf Step 2} \textit{There exists $\overline{\rho}>0$,
only depending on $\gamma$, $M$, $\mu$ and $M_0$, such that for
every $\rho$, $0<\rho\leq\overline{\rho}$, and for every
$x\in\Omega$ such that $s\rho<\hbox{dist}(x,\partial\Omega)\leq\frac{1}{4}$,
\begin{equation}
   \label{eq:step2_a}
   \frac{\int_{B_{r_{k(\rho)}}(x_{k(\rho)})}|\nabla^2 w_0|^2}
   {\int_\Omega|\nabla^2 w_0|^2}\leq C
   \left(\frac{\int_{B_{\rho}(x)}|\nabla^2 w_0|^2}
   {\int_\Omega|\nabla^2 w_0|^2}\right)
   ^{\delta_\chi^{k(\rho)-1}},
\end{equation}
\begin{equation}
   \label{eq:step2_b}
   \frac{\int_{B_{\rho}(x)}|\nabla^2 w_0|^2}
   {\int_\Omega|\nabla^2 w_0|^2}\leq C\left(\frac{\int_{B_{r_{k(\rho)}}(x_{k(\rho)})}|\nabla^2 w_0|^2}
   {\int_\Omega|\nabla^2 w_0|^2}\right)
   ^{\delta^{k(\rho)-1}},
\end{equation}
where $C>1$, $\delta\in(0,1)$, only depend on $\gamma$,
$M$ and $\mu$ whereas $\delta_\chi\in(0,1)$, only depends on
$\gamma$, $M$, $\mu$ and $M_0$}.
\begin{proof} [Proof of Step 2]
Let $\rho\leq\overline{\rho}=\min\{\rho_1,\rho_2,\rho_3, \rho_4\}$. Let us apply the three spheres inequality
\eqref{eq:3sph} to the discs of center $x_j$ and radii $r_j$, $3\chi
r_j$, $4\chi r_j$, for $j=1,... ,k(\rho)-1$. Since
$B_{r_{j+1}}(x_{j+1})\subset B_{3\chi r_j}(x_{j})$, for $j=1,...
,k(\rho)-1$, we have
\begin{equation}
   \label{eq:3sph1}
   \int_{B_{r_{j+1}}(x_{j+1})}|\nabla^2 w_0|^2
   \leq C
   \left(\int_{B_{r_j}(x_j)}|\nabla^2 w_0|^2\right)
   ^{\delta_\chi}
   \left(\int_{B_{4\chi r_j}(x_j)}|\nabla^2 w_0|^2\right)
   ^{1-\delta_\chi},
\end{equation}
with $C>1$ and $\delta_\chi$, $0<\delta_\chi<1$, only depending on
$\gamma$, $M$, $\mu$ and $M_0$ which we may rewrite as

\begin{equation}
\label{20.1} \frac{\int_{B_{r_{j+1}}(x_{j+1})}|\nabla^2 w_0|^2}
   {\int_\Omega|\nabla^2 w_0|^2}\leq C
   \left(\frac{\int_{B_{r_{j}}(x_j)}|\nabla^2 w_0|^2}
   {\int_\Omega|\nabla^2 w_0|^2}\right)
   ^{\delta_\chi}.
\end{equation}

{} By iterating \eqref{20.1} over $j=1,... ,k(\rho)-1$,
\eqref{eq:step2_a} follows. Similarly, by applying the three spheres
inequality to the discs $B_{r_j}(x_j),B_{3 r_j}(x_j),B_{4 r_j}(x_j)$
for $j=2,... ,k(\rho)$ and noticing that $B_{r_j}(x_{j-1})\subset
B_{3 r_j}(x_j)$ we can repeat the above argument obtaining
\eqref{eq:step2_b}.
\end{proof}

\medskip
\noindent {\bf Step 3} \textit{There exists $\rho^*$, only depending
on $\gamma$, $M$, $\mu$, $M_0$, $M_1$, $\delta_0$ and $F$, such that
for every $\rho\leq\rho^*$ and for every $\bar x\in \Omega_{s\rho}$
we have}

\begin{equation}
   \label{eq:x_to_y_case2}
   \frac{\int_{B_{\rho}(y)}|\nabla^2 w_0|^2}
   {\int_\Omega|\nabla^2 w_0|^2}\leq C\left(\frac{\int_{B_{\rho}(\overline{x})}|\nabla^2 w_0|^2}
   {\int_\Omega|\nabla^2 w_0|^2}\right)
   ^{\delta_{\chi}^{A_1+B_1\log\frac{1}{\rho}}},\qquad \forall y\in\Omega_{s\rho},
\end{equation}
\textit{where $C>1$, $B_1$ only depends on $\gamma$, $M$, $\mu$ and
$M_0$, whereas $A_1$ only depends on $\gamma$, $M$, $\mu$, $M_0$ and
$M_1$}.

\begin{proof} [Proof of Step 3]

First we consider the case $\bar x\in \Omega_{s\rho}$ satisfying
$\rm{dist}(\bar x,\partial\Omega)\leq\frac{1}{4}$. Let us take
$\rho\leq\bar\rho$. Since, by \eqref{eq:upper_lower_rho_k}, $5
r_{k(\rho)}\leq\frac{h_0}{4}$, it follows that $\Omega_{5
r_{k(\rho)}}$ is connected.

Let $y\in\Omega$ such that
$s\rho<\hbox{dist}(y,\partial\Omega)\leq\frac{h_0}{4}$ and let
$\sigma$ be an arc in $\Omega_{5 r_{k(\rho)}}$ joining $\bar
x_{k(\rho)}$ to $y_{k(\rho)}$. Let us define $\{x_i\}$, $i=1,...,L$,
as follows: $x_1=\bar x_{k(\rho)}$, $x_{i+1}=\sigma(t_i)$, where $t_i=\max\{t
\hbox{ s. t. }|\sigma(t)-x_i|=2 r_{k(\rho)}\}$ if
$|x_i-y_{k(\rho)}|>2 r_{k(\rho)}$, otherwise let $i=L$ and stop the
process. By construction, the discs $B_{r_{k(\rho)}}(x_i)$ are
pairwise disjoint, $|x_{i+1}-x_i|=2 r_{k(\rho)}$, for $i=1,...,
L-1$, $|x_L-y_{k(\rho)}|\leq {2 r_{k(\rho)}}$. Hence we have
\begin{equation}
   \label{eq:upper_L}
L\leq \frac{M_1}{\pi r_{k(\rho)}^2}.
\end{equation}
By an iterated application of the three spheres inequality
\eqref{eq:3sph} over the discs of center $x_i$ and radii $
r_{k(\rho)}$, $3 r_{k(\rho)}$, $4 r_{k(\rho)}$, we obtain
\begin{equation}
   \label{eq:x_k_to_y_k}
   \frac{\int_{B_{r_{k(\rho)}}(y_{k(\rho)})}|\nabla^2 w_0|^2}
   {\int_\Omega|\nabla^2 w_0|^2}\leq C
   \left(\frac{\int_{B_{r_{k(\rho)}}
   (\overline{x}_{k(\rho)})}|\nabla^2 w_0|^2}
   {\int_\Omega|\nabla^2 w_0|^2}\right)
   ^{\delta^L},
\end{equation}
where $C>1$ only depends on $\gamma$, $M$ and $\mu$.

By applying \eqref{eq:step2_a} for $x=\bar x$ and \eqref{eq:step2_b}
for $x=y$, we have

\begin{equation}
   \label{eq:x_to_x_k}
   \frac{\int_{B_{\rho}(y)}|\nabla^2 w_0|^2}
   {\int_\Omega|\nabla^2 w_0|^2}\leq C\left(\frac{\int_{B_{\rho}(\overline{x})}|\nabla^2 w_0|^2}
   {\int_\Omega|\nabla^2 w_0|^2}\right)
   ^{\delta_{\chi}^{k(\rho)-1}\delta^{k(\rho)+L-1}},
\end{equation}
where $C>1$ only depends on $\gamma$, $M$, $\mu$ and $M_0$.

The above estimate holds for every $y\in\Omega$ satisfying
$s\rho<\hbox{dist}(y,\partial\Omega)\leq\frac{{h_0}}{4}$. Next, let
$y\in\Omega$ satisfying
$\hbox{dist}(y,\partial\Omega)>\frac{{h_0}}{4}$. Since
$B_{5r_{k(\rho)}}(\bar x_{k(\rho)})\subset B_{\frac{1}{8}}(\widetilde{\bar x})\subset
\Omega$ we have

\begin{equation}
   \label{eq:x_k_far}
   \hbox{dist}(\bar x_{k(\rho)},\partial\Omega)\geq 5
   r_{k(\rho)},
\end{equation}
and by \eqref{eq:upper_lower_rho_k},
\begin{equation}
   \label{eq:y_far}
   \hbox{dist}(y,\partial\Omega)>\frac{{h_0}}{4}
   \geq 5 r_{k(\rho)}.
\end{equation}
Recalling that $\Omega_{5 r_{k(\rho)}}$ is connected, we can
consider an arc in $\Omega_{5 r_{k(\rho)}}$ joining $\bar
x_{k(\rho)}$ to $y$ and mimic the arguments just seen above over a
chain of $\tilde L$ discs of center $x_j\in\Omega_{5r_{k(\rho)}}$ and radii
$r_{k(\rho)}$, $3r_{k(\rho)}$, $4r_{k(\rho)}$, where
\begin{equation}
   \label{eq:upper_tilde_L}
\tilde L\leq \frac{M_1}{\pi r^2_{k(\rho)}}.
\end{equation}

By an iterated application of the three spheres inequality and by
applying \eqref{eq:step2_a} for $x=\overline{x}$ we have
\begin{equation}
   \label{eq:x_k_to_y_case2}
   \frac{\int_{B_{\rho}(y)}|\nabla^2 w_0|^2}
   {\int_\Omega|\nabla^2 w_0|^2}\leq C\left(\frac{\int_{B_{\rho}(\overline{x})}|\nabla^2 w_0|^2}
   {\int_\Omega|\nabla^2 w_0|^2}\right)
   ^{\delta_{\chi}^{k(\rho)-1}\delta^{\widetilde{L}}},
\end{equation}
where $C>1$ only depends on $\gamma$, $M$, $\mu$ and $M_0$. By
\eqref{eq:x_k_to_y_case2}, \eqref{eq:x_to_x_k}, \eqref{eq:upper_k},
\eqref{eq:upper_L}, \eqref{eq:upper_tilde_L} and since
$\delta_{\chi}<\delta$, we obtain \eqref{eq:x_to_y_case2}.

\medskip

Now let us consider the case $\bar x\in \Omega_{s\rho}$ satisfying
$\rm{dist}(\bar x,\partial\Omega)>\frac{1}{4}$. Let
$\rho\leq\overline{\rho}$ and notice that $B_{s\rho}(\bar x)\subset
B_{\frac{1}{16}}(\bar x)$. Hence, given any point $\tilde x$ such
that $|\bar x-\tilde x|=s\rho$, we have that $B_{\frac{1}{8}}(\tilde
x)\subset\Omega$. Therefore we can mimic the construction in Steps 1
and 2, finding a point $\bar x_{k(\rho)}\in\Omega_{5 r_{k(\rho)}}$,
with $k(\rho)$ satisfying \eqref{eq:lower_k}, \eqref{eq:upper_k} and
$r_{k(\rho)}$ satisfying \eqref{eq:upper_lower_rho_k}, such that the
following inequality holds
\begin{equation}
   \label{eq:x_to_x_k_step4}
   \frac{\int_{B_{r_{k(\rho)}}(\overline{x}_{k(\rho)})}|\nabla^2 w_0|^2}
   {\int_\Omega|\nabla^2 w_0|^2}\leq C
   \left(\frac{\int_{B_{\rho}(\overline{x})}|\nabla^2 w_0|^2}
   {\int_\Omega|\nabla^2 w_0|^2}\right)
   ^{\delta_{\chi}^{k(\rho)-1}}
\end{equation}
with $C>1$ only depending on $\gamma$, $M$ and $M_0$.

Let $y\in\Omega_{s\rho}$ such that
$\hbox{dist}(y,\partial\Omega)\leq\frac{1}{4}$. By the same
arguments seen above, we have
\begin{equation}
   \label{eq:x_to_y_case1_step4}
   \frac{\int_{B_{\rho}(y)}|\nabla^2 w_0|^2}
   {\int_\Omega|\nabla^2 w_0|^2}\leq C
   \left(\frac{\int_{B_{\rho}(\overline{x})}|\nabla^2
   w_0|^2}{\int_\Omega|\nabla^2 w_0|^2}\right)
^{\delta_\chi^{k(\rho)-1}\delta^{k(\rho)+L-1}},
\end{equation}
where $C>1$ only depends on $\gamma$, $M$, $\mu$, $M_0$, and $L$
satisfies \eqref{eq:upper_L}.

Let $y\in\Omega_{s\rho}$ such that
$\hbox{dist}(y,\partial\Omega)>\frac{1}{4}$. By repeating the arguments above, we have
\begin{equation}
   \label{eq:x_to_y_case2_step4}
   \frac{\int_{B_{\rho}(y)}|\nabla^2 w_0|^2}
   {\int_\Omega|\nabla^2 w_0|^2}\leq C
   \left(\frac{\int_{B_{\rho}(\overline{x})}|\nabla^2
   w_0|^2}{\int_\Omega|\nabla^2 w_0|^2}\right)^{\delta_{\chi}^{k(\rho)-1}\delta^{\widetilde{L}}},
\end{equation}
where $\widetilde{L}$ satisfies \eqref{eq:upper_tilde_L} and $C>1$
only depends on $\gamma$, $M$, $\mu$ and $M_0$.

{}From \eqref{eq:x_to_y_case1_step4}, \eqref{eq:x_to_y_case2_step4},
\eqref{eq:upper_k}, \eqref{eq:upper_L}, \eqref{eq:upper_tilde_L},
and recalling that $\delta_\chi<\delta$, we obtain
\eqref{eq:x_to_y_case2}.

Let us cover $\Omega_{(s+1)\rho}$ with internally nonoverlapping
closed squares of side $l=\frac{2\rho}{\sqrt{2}}$. Any such square is
contained in a disc of radius $\rho$ and center at a point of
$\Omega_{s\rho}$ and the number of such squares is dominated by
\begin{equation}
   \label{eq:upper_N}
N=\frac{M_1}{2\rho^2}.
\end{equation}

Therefore, {}from \eqref{eq:x_to_y_case2} and \eqref{eq:upper_N}, we
have
\begin{equation}
   \label{eq:reverse}
\int_{B_{\rho}(\overline{x})}|\nabla^2 w_0|^2 \geq
\int_\Omega|\nabla^2 w_0|^2
\left(\frac{C'\rho^2\int_{\Omega_{(s+1)\rho}}|\nabla^2
w_0|^2}{\int_\Omega|\nabla^2
w_0|^2}\right)^{\delta_{\chi}^{-A_1-B_1\log\frac{1}{\rho}}},
\end{equation}
where $B_1$ and $C'>0$ only depend on $\gamma$, $M$, $\mu$ and $M_0$, whereas $A_1$
only depends on $\gamma$, $M$, $\mu$ and $M_0$ and $M_1$.

By Lemma \ref{lem:Frequency}, assuming also
$\rho\leq\frac{\widetilde{\rho}}{s+1}$, where $\widetilde{\rho}$ has
been introduced in Lemma \ref{lem:Frequency} and only depends
$\gamma$, $M$, $\mu$, $M_0$, $M_1$, $\delta_0$, $F$ we have
\begin{equation}
   \label{eq:reverse_last}\int_{B_{\rho}(\overline{x})}|\nabla^2 w_0|^2
   \geq \left(\widetilde{C}\rho^2\right)^{\delta_{\chi}^{-A_1-B_1\log\frac{1}{\rho}}}\int_\Omega|\nabla^2
   w_0|^2,
\end{equation}
where $\tilde C>0$ only depends on $\gamma$, $M$, $\mu$, $M_0$,
$M_1$ and $\delta_0$. Let us take $\rho\leq\widetilde{C}$. Noticing
that $|\log \rho|\leq \frac{1}{\rho}$, for every $\rho>0$, and that
$\tilde\rho<1$, by straightforward computations we obtain that
\eqref{eq:LPS} holds with $A=3\exp(A_1|\log\delta_\chi|)$,
$B=|\log\delta_\chi|B_1+1$ for every $\rho\leq\rho^*$ with
$\rho^*=\min\{\bar\rho, \frac{\tilde\rho}{s+1},\tilde C\}$, $\rho^*$ only
depending on $\gamma$, $M$, $\mu$, $M_0$, $M_1$, $\delta_0$, and
$F$.
\end{proof}
{\bf Conclusion} We have seen that \eqref{eq:LPS} holds for every
$\rho\leq\rho^*$ and for every $\bar x\in\Omega_{s\rho}$, where
$\rho^*$ only depends on $\gamma$, $M$, $\mu$, $M_0$, $M_1$,
$\delta_0$ and $F$.

If $\rho>\rho^*$ and $\bar
x\in\Omega_{s\rho}\subset\Omega_{s\rho^*}$, then we have
\begin{equation}
   \label{eq:rho>}
   \int_{B_{\rho}(\bar x)}|\nabla^2 w_0|^2
   \geq \int_{B_{\rho^*}(\bar x)}|\nabla^2 w_0|^2
   \geq C^*\int_{\Omega}|\nabla^2 w_0|^2,
\end{equation}
where $C^*$ only depends on  $\gamma$, $M$, $\mu$, $M_0$, $M_1$, $\delta_0$ and
$F$.
Since $\bar x\in \Omega_{s\rho}$, we have that
\begin{equation}
\label{diam_lower}
 \hbox{diam}(\Omega)\geq 2s\rho,
\end{equation}
and, on the other hand,
\begin{equation}
\label{diam_upper}
 \hbox{diam}(\Omega)\leq C_2,
\end{equation}
with $C_2$ only depending on $M_0$ and $M_1$, so that
\begin{equation}
   \label{eq:upper_for_rho>}
   \frac{2s}{C_2}\leq\frac{1}{\rho}.
\end{equation}
By \eqref{eq:rho>} and \eqref{eq:upper_for_rho>}, we have
\begin{equation}
   \label{eq:rho>finale}
   \int_{B_{\rho}(\bar x)}|\nabla^2 w_0|^2
   \geq \frac{C}{\exp\left[A\left(\frac{1}{\rho}\right)^B\right]}
   \int_{\Omega}|\nabla^2 w_0|^2,
\end{equation}
with $C=C^*\exp\left[A\left(\frac{2s}{C_2}\right)^B\right]$.

\end{proof}

%
\bibliographystyle{alpha}
%
%

\end{document}